\documentclass{article}

\usepackage{amsfonts}
\usepackage{float}
\usepackage{amsmath}
\usepackage{amssymb}
\usepackage{amsthm}
\usepackage{graphicx}
\usepackage{multirow}
\usepackage{color}
\usepackage{ulem}
\usepackage{setspace}
\usepackage{caption}
\usepackage{subcaption}
\usepackage{rotating}
\usepackage{threeparttable}
\usepackage{url}
\usepackage{framed}

\usepackage[round]{natbib}

\newcommand{\mbf}[1]{\mbox{\boldmath $#1$}}

\def\bA{\mathbf{A}}

\def\bP{\mathbf{P}}
\def\bR{\mathbf{R}}

\def\bu{\mathbf{u}}

\def\bW{\mathbf{W}}
\def\bX{\mathbf{X}}

\def\bG{\mathbf{G}_{n}}

\def\bD{{\mathbf D}}

{\catcode `\@=11 \global\let\AddToReset=\@addtoreset}
\AddToReset{equation}{section}

\AddToReset{Theorem}{section}

\newtheorem{lem}{Lemma}[section]
\newtheorem{thm}{Theorem}[section]

\theoremstyle{remark}
\newtheorem{rem}{Remark}[section]

\newcommand{\cB}{{\cal B}}

\newcommand{\cD}{{\cal D}}

\newcommand{\cL}{{\cal L}}

\newcommand{\cN}{{\cal N}}

\newcommand{\cQ}{{\cal Q}}

\newcommand{\cW}{{\cal W}}

\def\bc{\begin{center}}
\def\bd{\begin{description}}
\def\be{\begin{enumerate}}
\def\ec{\end{center}}
\def\ed{\end{description}}

\def\ee{\end{enumerate}}
\def\ben{\begin{equation}}
\def\benn{\begin{equation*}}
\def\een{\end{equation}}
\def\eenn{\end{equation*}}
\def\benr{\begin{eqnarray}}
\def\eenr{\end{eqnarray}}
\def\benrr{\begin{eqnarray*}}
\def\eenrr{\end{eqnarray*}}

\def\al{\alpha}
\def\b{\beta}

\def\del{\delta}
\def\edt{\end{document}}

\def\G{\boldsymbol{\Gamma}}

\def\iny{\infty}

\def\la{\lambda}

\def\noi{\noindent}

\def\r{\ref}

\def\ra{\rightarrow}

\def\si{\sigma}

\def\sti{\sum_{i=1}^n}

\def\stj{\sum_{j=1}^p}
\def\stk{\sum_{k=1}^n}

\def\vep{\varepsilon}
\def\eps{\epsilon}

\def\wh{\widehat}

\def\wt{\widetilde}

\def\R{{\mathbb R}}

\def\bxi{\mbf{x}_{i}}
\def\bSn{\mbf{S}_{n}}
\def\bqi{\mbf{q}_{i}}

\def\B{\boldsymbol{\beta}}
\def\W{\boldsymbol{\cW}}
\def\S{\boldsymbol{\Sigma}}

\def\bet{\boldsymbol{\eta}}
\def\bbxi{\boldsymbol{\xi}}

\def\bL{\boldsymbol{\Lambda}}
\def\bO{\boldsymbol{\Omega}}

\def\tbf{\textbf}

\def\mE{\mathbb{E}}
\def\mP{\mathbb{P}}
\def\mR{\mathbb{R}}

\newcounter{Qcounter}
\newcounter{show}
\newcommand{\inhy}[1]{\int{#1}dH(y) }

\title{Estimation of the rate parameter of the probability distribution on the regression setup}
\author{Jiwoong Kim\\
University of South Florida}

\usepackage{setspace}

\begin{document}
\maketitle
\begin{abstract}
When the rate parameter of the exponential distribution is associated with predictors, then main interest will be how to estimate the regression parameter. In this paper, we will investigate how to estimate the parameter on the regression setup of the exponential distribution. To that end, we propose a new estimator, and its asymptotic properties will be discussed.
\end{abstract}
\noi

\section{Introduction}

In the literature on statistical theory and probability distributions, the exponential distribution is popular for several reasons. For example, it will not be an exaggeration to claim that its distribution function is, de facto, the second simplest, being ranked after that of the uniform distribution. Consequently, this feature of simplicity makes the exponential distribution more practical and useful, and the domain of its application
has quickly expanded to many other disciplines by addressing real-world problems.
Among those disciplines, the exponential distribution paved the way and laid the foundation for survival analysis, which models the time to the occurrence of a specific event. Having been the most popular with medical science, survival analysis has also drawn attention from many non-medical disciplines, such as social science, engineering, and physics. In survival analysis, \textit{hazard rate} -– which indicates the frequency of deaths -– can be expressed as $f/(1-F)$ with $f$ and $F$ being density and distribution functions, respectively.


From the constant rate $\lambda$, survival analysis derives two critical curves used to understand a population's lifespan. The probability density function
measures the absolute probability of the event occurring at exactly time $t$. It drops off sharply over time, showing that as time goes on, fewer individuals remain alive to experience the event. On the contrary, the survival function 
measures the probability that a subject will survive longer than time $t$. It starts at 1 (100\% survival at time zero) and decays smoothly toward 0, which shows the relationship between exponential distribution and the survival distribution.

An interesting questions arises when the unknown rate parameter is associated with some predictors, which extends the one sample exponential distribution to the regression setup. Well-known example will be the Cox proportional hazards (Cox-PH, or simply Cox) model, which assumes the constant rate $\la$ over time. While real-world biological systems or mechanical parts often wear out over time—violating the constant hazard assumption—the exponential distribution remains the vital baseline model. It acts as the mathematical benchmark from which more complex survival models. Minimum distance (MD) method is known to possess many desirable properties, such as asymptotic normality and robustness of the resulting estimator. As \cite{Koul2002} extended the MD methodologies from estimation of the location parameter of one sample to that of the parameter on regression setup. Being analogous to what he had done, we will extend the MD method from the single rate parameter estimation to the regression parameter estimation of the Cox proportional hazard model.


\section{Minimum distance estimation}\label{Sec:MDE}
\subsection{Literature review}
During the 1970s and 1980s, many statisticians -- e.g., \cite{Koul1970}, \cite{Millar1984}, \cite{Donoho88a}, and \cite{Donoho88b} -- have conducted research on the MD estimation since it was proposed by \cite{Wolfowitz1953}: see also references in \cite{Koul2002}. The distance function measures the discrepancy between the observed random sample and the assumptions underlying the theories. More specifically, the distance function –- which contains the parameter of interest as an argument -- computes the numeric difference between the empirical function and the modeled function, constructed from observed data and assumptions, respectively. Then, the MD estimation method, as the name implies, seeks the optimal value that minimizes the distance function. Researchers investigated the resulting estimator after employing various distance functions. For example, \cite{Beran} used the Hellinger distance using empirical and modeled density functions. Among many research works on the distance function in the literature of the MD estimation, \cite{Parr} demonstrated that the MD estimator obtained from the Cramer von-Mises (CvM) distance function exhibits better robustness than those obtained from other distance functions.

In the past two decades, however, no more rigorous research has been conducted; only a few studies have further investigated MD estimation. \cite{Kim2018} proposed a novel algorithm to compute the MD estimator, while \cite{Kim2020} demonstrated that the MD estimator maintains the desirable asymptotic properties under the assumption of independent observations even when independence doesn’t hold. Applying the MD method to a discrete distribution, \cite{Kim2026} demonstrated that the MD estimator still retains asymptotic properties and robustness, thereby comparing favorably with other well-celebrated estimators, including the ML estimator.

One of the fundamental reasons the popularity of the MD method has been waning rapidly is the complexity of its distance function. The empirical distribution function, which is a collection of indicator functions of the observed sample and the parameter of interest, is the main culprit obstructing the search for the optimal solution to the distance function. Since the indicator function is not smooth, it is not differentiable with respect to the parameter; unlike the ML estimation, the closed-form expression for the solution does not exist. Therefore, research on MD estimation should rely on computationally expensive numerical methods to solve the optimization problem. \cite{Kim2026} modified the distance function so that the empirical distribution function contains only the observed sample, while the modeled distribution function contains the parameter of interest. As a result, the distance function became smooth and differentiable with respect to the parameter, thereby reducing the computational cost. 
Findings in \cite{Kim2026}, however, are limited in that the modified distance function is valid only for a discrete probability distribution, the estimation problem of which is presumed to be less challenging than that of a continuous probability distribution function. In this article, we extend his approach to estimating the regression parameter of this study.

\subsection{The distance function of MD estimation}\label{Sec:CvM_distance_function}
Let $f$ and $F$ denote the probability density and distribution functions with a rate parameters, respectively. For example, $f$ and $F$ of the exponential distribution with a rate parameter $\la$ will be $f(y;\la):=\la e^{-\la y}$ and $F(y;\la):=1-e^{-\la y}$, respectively, for real $y\in \mathbb{R}$, while those of the Wiebull distribution with a known shape parameter $\al$ will have $f(y;\la)=\la\al (\la y)^{\al-1}e^{-(\la y)^{\al}}$ and $F(y;\la) = 1-e^{(\la y)^{\al}}$. Let $g(y,\la):=\partial F(y;\la)/\partial \la$. Note that $f(y;\la) = g(\la,y)$, and hence, $f$ and $g$ share many common features such as smoothness and differentiability. In the literature of MD estimation theories, $f$ has played in the crucial role, such as determining the asymptotic variance of the resulting MD estimator; however, in this study, $g$ will take the role of $f$.

Consider a sample of independent observations, $Y_{1},...,Y_{n}$ whose distribution function is $F$ with different rate parameters $\la_{1},...\la_{n}$. Assume that the rate parameter $\la_{i},\,\,1\le i \le n$ are exponentially associated with a predictor $\mbf{x}_{i}\in \mathbb{R}^{p}$, that is, $\la_{i}=e^{\mbf{x}_{i}'\B_{0}}$. Real-world examples of statistical models on this setup include survival analysis, as mentioned in the introduction, where $Y_{i}$ represents survival time of the $i$th individual, while a hazard rate, $\la_i$, is affected by the predictors of the individual, including age, income, etc. To estimate the parameter $\B_{0}\in \mathbb{R}^{p}$, we first define the distance function $\cL(\B)$ for $\B\in \mathbb{R}^{p}$ with an integrating measure $H$ as follows:
\ben\label{eq:distance_function}
\cL(\B) = \stj\int\left[ \sti d_{ij}\Big\{\textrm{I}(Y_{i}\le y)-F(y;\mbf{x}_{i}'\B) \Big\} \right]^{2}\,dH(y),
\een
where $\textrm{I}(\cdot)$ is an indicator function, and $d_{ij}\in \mathbb{R},\,1\le i\le n,\,1\le j\le p$ are real numbers. Note that the distance function can be adjusted through the choice of the weights $d_{ij}$, which is one of the main merits of the MD estimator. For example, \cite{Koul2002} demonstrated that proper choice of the weights will give a rise to the most efficient estimator for the given $F$. In addition, \cite{Kim2026} used the uniform weights and showed that the distance function can be stabilized against certain impact from outliers and hence yield more robust estimator than other methods. As will be shown later, choice of $d_{ij}$ that yields the optimal result will not be fortuitous in this study; rather the weight will be methodically selected.

Next, we define the MD estimator belonging to the neighborhood of the true parameter. More precisely, the MD estimator will solve the following optimization problem
\ben\label{eq:opt}
\cL(\widehat{\B}) = \inf_{\B\in\mathbb{R}^{p}}\cL(\B).
\een
where infimum is taken over $\cN_{b}(\B_{0}):=\{\B\in\mathbb{R}^{p}: \bA^{-1}|\B-\B_{0}|\le b<\iny\}$ for some $p\times p$ matrix $\bA$.

Unfortunately, the optimal solution to the above problem doesn't have any closed-form expression, which is the most stumbling block to finding the MD estimator. However, the  concept that the distance function will be \textit{uniformly locally asymptotically quadratic} (ULAQ) over $\cN(\B_{0})$ ensures that the MD estimator can be searched by numerical methods using its gradient. Furthermore, it possesses the desirable properties, including asymptotic normality and robustness. In the literature of the MD estimation, it is not exaggeration to state that the ULAQ conditions play the most crucial role to obtain the MD estimator and its asymptotic properties; the successful investigation of the MD estimation hinges on satisfying the ULAQ conditions. The assumptions (\tbf{U1})-(\tbf{U3}) below state the ULAQ conditions required for this study.
\begin{itemize}
  \item[\textbf{(U.1)}] There exist a sequence of $p\times 1$ random vector $\mbf{S}_{n}(\B_{0})$ and a sequence of $p\times p$ real matrix $\bW_{n}(\B_{0})$
   such that for all $0<b<\iny$
   \benn
   \sup_{|\del_{n}(\b-\b_{0})|\leq b}\left|\cL(\B)- \cL(\B_{0})-2(\B-\B_{0})'\mbf{S}_{n}(\B_{0})-(\B-\B_{0})'\bW_{n}(\b_{0})(\B-\B_{0})\right|=o_{p}(1).
   \eenn
  \item[\textbf{(U.2)}] For all $\vep>0$, there is a $0<z_{\vep}<\iny$ such that
  \benn
    \mP\big(|\cL(\B_{0})|\leq z_{\vep} \big)\geq 1-\vep.
  \eenn
  \item[\textbf{(U.3)}] For all $\vep>0$ and $0<c<\iny$, there is a $0<b<\iny$ and $N$ -- both depending on $\vep$ and $c$ -- such that
  \benn
    \mP\Big( \inf_{\del_{n}|\B-\B_{0}|>b} |\cL(\B)|> c \Big)\geq 1-\vep, \qquad \textrm{ for all }n\geq N.
  \eenn
\end{itemize}
The following lemma reproduces Theorem 5.4.1 from \cite{Koul2002}, which gives a clue for the asymptotic distribution of the resulting MD estimator that satisfies the ULAQ conditions.
\begin{lem}\label{lem:koul_asym_distr_Sn}
Assume that $\cL$ satisfies the ULAQ assumptions (\tbf{U.1})- (\tbf{U.3}). Let $\cB_{n}:=\bA\bW_{n}(\B_{0})\bA$ where $\bA$ is used to define $\cN_{b}(\B_{0})$. Let $\wh{\B}$ denote the MD estimator that solves the optimization problem in (\ref{eq:opt}). Then,
\benn
\cB_{n}\bA^{-1} (\wh{\B}-\B_{0}) = -\bA\mbf{S}_{n}(\B_{0})+o_{p}(1).
\eenn
\end{lem}
\noi
The lemma says the asymptotic distribution of the MD estimator will be determined by that of $\bA\mbf{S}_{n}$; thus, the asymptotic normality of $\bA\mbf{S}_{n}$ -- which will be proven in the next section -- will vouch for that of the MD estimator.

\subsection{MD estimator of $\B$ and its asymptotical properties}
\setcounter{Qcounter}{0}
Consider $n$ pairs of observations, $(Y_{1}, \mbf{x}_{1}'),...,(Y_{n}, \mbf{x}_{n}')$ where $\mbf{x}_{i}\in \mR^{p}$. Let $\bX$ be an $n\times p$ design matrix whose $i$th row vector is $\mbf{x}_{i}'$.
Using $d_{ij}$, $1\le i\le n,\,1\le j\le p$, the same weights in (\ref{eq:distance_function}), define an $n\times p$ matrix $\bD:=((d_{ij}))$. As mentioned in the previous section, $\bD$ will be selected after the meticulous investigation so that we can obtain the optimal MD estimator. It is a matter of course that $\bD$ will be related with the design matrix $\bX$ as demonstrated in other studies. For example, \cite{Koul2002} proved that $\bD \varpropto \bX$ -- that is, $\bD$ is $\bX$ multiplied by a non-singular $p\times p$ matrix obtained from $\bX$ -- will yield the most efficient estimator when estimating the regression parameter of linear models with independent observations. As shown later, the optimal $\bD$ of this study turns out to be much more complicating than that of \cite{Koul2002}.

Let $\bqi(\cdot,\B):=\partial F(\cdot;\bxi'\B)/\partial \B$. For the exponential $F$, we have $\bqi(y, \B) =  ye^{-y\bxi'\B}e^{\bxi'\B}\bxi$. Recall $\la_i=e^{\bxi'\B}$ and $g(y,\la)=ye^{-\la y}$. Using $\la_{i}$ and $g$ will simplify $\bqi(y,\B) =g(y,\la_{i})\la_{i}\bxi$. To denote that $g(\la_{i};y)\la_{i}$ is indeed a function of $\B$, let $g(\la_{i};y)\la_{i}:=g_{i}(y,\B)$. With all these notations, we are ready to state the necessary assumptions for this study. It should be admitted that most of the assumptions have a root in \cite{Koul2002}.
\begin{itemize}
\item[(\tbf{a.\addtocounter{Qcounter}{1}\theQcounter})] Let $\textbf{B}$ denote an $n\times n$ symmetric, positive definite matrix. Then, $\bX'\textbf{B}\bX$ is nonsingular. In addition, there exists a $p\times p$ nonsingular matrix $\bA$ such that $\bA = (\bX'\textbf{B}\bX)^{-1/2}$.
\item[(\tbf{a.\addtocounter{Qcounter}{1}\theQcounter})] For all $1\le j\le p$, the following hold true: $\sti d_{ij}^{2}=1$ and $\max_{1\le i\le n}d_{ij}=o(1)$.
\item[(\tbf{a.\addtocounter{Qcounter}{1}\theQcounter})] Let $\mbf{c}_{i}:=\bA \mbf{x}_{i}$. Then $\max_{1\le i\le n}\|\mbf{c}_{i}\|=o(1)$.
\item[(\tbf{a.\addtocounter{Qcounter}{1}\theQcounter})] For $1\le j\le p$, $\sti \|d_{ij}\mbf{c}_{i}\| = O(1)$.
\item[(\tbf{a.\addtocounter{Qcounter}{1}\theQcounter})]
    Let $\la_{i}^{0}:=\bxi'\B_{0}$. With $g(y,\la)=\partial F(y;\la)/\partial \la$, the following holds:
    \benn
    \max_{1\le i\le n} \sup\int_{0}^{\iny}\{\la(1-y\la)g(y,\la)\}^2\,\,dH(y)<\iny,
    \eenn
    where the supremum is taken over $\la$ between $\la_{i}^{0}$ and $\la_{i}$.
\item[(\tbf{a.\addtocounter{Qcounter}{1}\theQcounter})] Consider $F(\cdot;\bxi'\B)$, the df of the random variable $Y_{i}$. Then,
    \benn
    \stj\inhy{\sti d_{ij}^{2} F(y;\bxi'\B)\{1-F(y;\bxi'\B)\}}=O(1).
    \eenn

\item[(\tbf{a.\addtocounter{Qcounter}{1}\theQcounter})] Let $\bG(y,\B)$ be an $n\times n$ diagonal matrix whose $i$th entry is $g_{i}(y,\B)$. Then a $p\times p$ matrix $\G_{n}(y,\B):=\bD'\bG(y,\B)\bX\bA$ is nonsingular for $y\in \mR$ and $\B\in \mR^{p}$.

\item[(\tbf{a.\addtocounter{Qcounter}{1}\theQcounter})] Let $\boldsymbol{\gamma}_{j}(y,\B)\in \mR^{p},\,1\le j\le p$ denote the $j$th column vector of $\G_{n}(y,\B)$. Then the following is true:
    \benn
    \stj\inhy{\|\boldsymbol{\gamma}_{j}(y,\B)\|^r}<\iny, \quad \textrm{ for }r=1,2.
    \eenn

\item[(\tbf{a.\addtocounter{Qcounter}{1}\theQcounter})] Let $\mbf{e}\in \mathbb{R}^{p}$ be a unit vector, that is, $\|\mbf{e}\|=1$. With $\G_{n}(y,\B)$, let $\G_{H}:=\inhy{ \G_{n}(y,\B)l(y)}$ where $l:\mR\ra\mR$ is such that $\inhy{l^{2}(y)}<\iny$. Let $k_{n}(\mbf{e}):=\mbf{e}'\G_{H}\mbf{e}$. Then there exists an $\alpha>0$ such that
    \benn
    \liminf_{n}\big\{\inf\{k_{n}(\mbf{e}):\mbf{e}\in \mathbb{R}^{J}\}\big\}\ge \alpha.
    \eenn
\item[(\tbf{a.\addtocounter{Qcounter}{1}\theQcounter})] For all $1\le k\le n$ and for all unit vectors $\mbf{e}\in \R^{p}$,  either $\mbf{d}_{k}'\mbf{e}\mbf{x}_{k}'\bA\mbf{e}\ge 0$ or $\mbf{d}_{k}'\mbf{e}\mbf{x}_{k}'\bA\mbf{e}\le 0$ holds true.

\end{itemize}
\begin{rem}
For the exponential $F$, $g(y,\la) = ye^{-\la y}$. With $H(y)\equiv y$, the integral of the equation in \tbf{(a.5)} will be simplified to
\benn
\int_{0}^{\iny}\{\la(1-y\la)g(y,\la)\}^2\,\,dy\le 2\int_{0}^{\iny}\la^2(1+y^2\la^2)y^2e^{-2\la y}\,\,dy= \frac{4}{\la},
\eenn
where the inequality readily follows from $(a-b)^{2}\le 2(a^2+b^2)$ for real $a,b\in \mR$, and hence, the assumption \tbf{(a.5)} is equivalent to $\min_{1\le i\le n}\{\la_{i}^{0},\la_{i}\}>0$. Using the probability df for the integrating measure, for example, $H(y)=F(y)$, the left-hand side (LHS) of the equation of the claim will be bounded by $14/81$, and hence, the assumption will be trivially met, regardless of $\la$.
\end{rem}
To conserve the space, let $F_{i}(\cdot;\B): = F(\cdot;\mbf{x}_{i}'\B)$. Define $\W:=(\cW_{1},...,\cW_{j})'\in \mathbb{R}^{p}$ where the $j$ th entry is
\benn
\cW_{j}(y, \B):= \sti d_{ij}\Big\{\textrm{I}(Y_{i}\le y)-F_{i}(y;\B) \Big\}.
\eenn
Note that the distance function can be written as
\benn
\cL(\B)=\stj \inhy{\cW_{j}(y, \B)^2} = \inhy{\W(y, \B)'\W(y, \B)}.
\eenn
Next, we shall verify that the following $\bSn$ and $\bW_{n}$ satisfy the ULAQ conditions:
\benn
\bSn(\B)=-\stj \inhy{\cW_{j}(y,\B)\sti d_{ij}\bqi(y,\B)},\,\,
\bW_{n}(\B)= \stj \inhy{ \sti\stk d_{ij}d_{kj}\bqi(y,\B)\mbf{q}_{k}(y,\B)'}.
\eenn
Subsequently, define a quadratic function
\benn
\cQ(\B)=\cL(\B_{0})+2(\B-\B_{0})'\bSn(\B_{0})+(\B-\B_{0})'\bW_{n}(\b_{0})(\B-\B_{0}).
\eenn
Recall $\bG$ and $\G_{n}$ from the assumption \tbf{(a.7)}. Replacing $\bqi(\B;y)$ with $g_{i}(\B;y)\bxi$, rewrite $\bSn$ and $\bW_{n}$ in a matrix form as follows:
\ben\label{eq:SandW_Matrix}
\bSn(\B) = \bA^{-1}\inhy{ \G_{n}(y,\B)'\W(y,\B)}, \quad \bW_{n}(\B) = \bA^{-1}\inhy{ \G_{n}(y,\B)'\G_{n}(y,\B)}\bA^{-1}.
\een
Recall $\cN_{b}(\B_{0})=\{\B\in\mathbb{R}^{p}: \bA^{-1}|\B-\B_{0}|\le b<\iny\}$ and the ULAQ conditions. The first ULAQ condition implies the distance function $\cL$ can be uniformly approximated by the quadratic function $\cQ$ over $\cN_{b}(\B_{0})$, which is demonstrated by Theorem \ref{thm:ulaq1}. Before proceeding to the theorem, we shall prove the next lemma, which will be used for the proof of the theorem.
\begin{lem}\label{lem:mean_value_theorem_F}
For $0<b<\iny$,
\benn
\sup_{\B\in\cN_{b}(\B_{0})} \stj \inhy{\left[\sti d_{ij}\left\{F_{i}(y;\B)-F_{i}(y; \B_{0})-(\B-\B_{0})'g_{i}(y,\B_{0})\bxi\right\}\right]^{2}} =o(1).
\eenn
\end{lem}
\begin{rem}
The above lemma is analogue of the assumption (i) of Section 5.5 from \cite{Koul2002}. The difference between two originates from the fact that $f_{i}$ is replaced by $g_{i}$, as mentioned earlier.
\end{rem}
\begin{proof}
Let $\mbf{u} = \bA^{-1}(\B-\B_{0})$. Note that
$F_{i}(y;\B)-F_{i}(y;\B_{0})=(\B-\B_{0})'\bxi g_{i}(y,\wt{\B})$,
where $\wt{\B}=c\B+(1-c)\B_{0}$ for some $c\in (0,1)$. Let $\wt{\la}_{i}=\bxi'\wt{\B}$ and $\la_{i}^{0}=\bxi'\B_{0}$.
Also, observe that
\benn
|g_{i}(y,\wt{\B})-g_{i}(y,\B_{0})|\le \left|(\wt{\B}-\B_{0})'\frac{\partial g_{i}(y,\B)}{\partial \B}\right|\le \|\bu\|\cdot \|\bA \bxi\|\cdot |\la_{i}^{*}(1-y\la_{i}^{*})g(\la_{i}^{*};y)|,
\eenn
where $\la_{i}^{*}$ lies between $\la_{i}^{0}$ and $\wt{\la}_{i}$; the mean value theorem readily implies the first inequality, while the second inequality follows from $|\bxi'(\wt{\B}-\B_{0})|\le |\bxi'(\B-\B_{0})|$. Recall $\mbf{c}_{ni}$ from the assumption \tbf{(a.2)}. Finally, for $\|\bu\|\le b$,
\benrr
(\textrm{LHS of the equation})
&\le& \sup_{\|\bu\|\le b}\stj\inhy{\left[ \sti d_{ij}\bu'\mbf{c}_{ni}\{g_{i}(y,\wt{\B})-g_{i}(y,\B_{0})\} \right]^2},\\
&\le& p\cdot b^4\left(\sti \|d_{ij}\mbf{c}_{ni}\|\right)^2 \max_{1\le i\le n} \sup_{\la}\int_{0}^{\iny}\{\la(1-y\la)g(\la;y)\}^2\,\,dH(y)
    \rightarrow0,
\eenrr
where the convergence to 0 follows from \tbf{(a.2)}-\tbf{(a.5)}, thereby completing the proof of the lemma.
\end{proof}
Consider a vector-valued function $\mbf{a}(y):=(a_{1}(y),...,a_{p}(y))\in \mR^{p}$ whose entries are functions of $y\in \mR$, that is, $a_{j}:\mR\ra\mR$ for all $1\le j\le p$. Let $\|\mbf{a}\|_{H}^{2}$ denote its $L_{2}$ norm
\benn
\|\mbf{a}\|_{H}^{2} := \inhy{\mbf{a}'(y)\mbf{a}(y)}=\stj\inhy{ a_{j}(y)^2}.
\eenn
Let $\mbf{\kappa}:=(\kappa_{1},...,\kappa_{p})'\mR^{p}$ whose $j$th entry is the integrand of the equation in the above lemma, that is,
\benn
\kappa_{j}(y,\B):= \sti d_{ij}\left\{F_{i}(y;\B)-F_{i}(y; \B_{0})-g_{i}(y,\B_{0})(\B-\B_{0})'\bxi\right\},
\eenn
and hence, the lemma can be written as $\sup_{\|\bu\|\le  b} \|\mbf{\kappa}\|_{H}^{2}=o_{p}(1)$. In what follows, we will use the $L_{2}$ norm notation to conserve space. Next theorem serves to demonstrate that the first ULAQ condition holds.
\begin{thm}\label{thm:ulaq1}
Suppose that assumptions \tbf{(a.1)}-\tbf{(a.8)} hold. Then, the distance function $\cL$ in (\ref{eq:distance_function}) satisfies \textbf{(U.1)}, that is, for any $0<b<\iny$,
\benn
   \mathbb{E}\Big(\sup|\cL(\B)- \cQ(\B)|\Big)=o(1),
\eenn
where the supremum is taken over $\cN_{b}(\B_{0})$.
\end{thm}

\begin{proof} Recall $\bG$ and $\G_{n}$ from \tbf{(a.7)}, and rewrite $\W=\W(y,\B)$ and $\W_{0}=\W(y,\B_{0})$. Note that
\benn
\W = \W_{0}-(\B-\B_{0})'\bX'\bG\bD-\mbf{\kappa},
\eenn
and hence, with $\bu=\bA^{-1}(\B-\B_{0})$,
\benrr
|\cL(\B) - \cQ(\B)|&=& \|\W\|_{H}^{2} - \|\W_{0}-(\B-\B_{0})'\bX'\bG\bD\|_{H}^{2}\\
&\le& \|\mbf{\kappa}\|_{H}^{2}+2\|\mbf{\kappa}\|_{H}\Big[
\|\W_{0}\|_{H}+\|\bu'\G_{n}'\|_{H}
\Big]\longrightarrow 0,
\eenrr
where the inequality follows from applying the Cauchy-Schwarz inequality to the cross product terms after expanding $\|\W\|_{H}^{2}$. Note that \tbf{(a.6)} and \tbf{(a.8)} readily imply $\|\W_{0}\|_{H}<\iny$ and
\benn
\sup_{\|\bu\|\le b}\|\bu'\G_{n}\|_{H}^{2}\le b^2\stj\inhy{\|\mbf{\gamma}_{j}(y)\|^2}<\iny,
\eenn
respectively. Finally,  $\|\mbf{\kappa}\|_{H}^{2}=o(1)$ from Lemma \ref{lem:mean_value_theorem_F} completes the proof of the theorem.
\end{proof}
\noi
\begin{lem}\label{lem:asym_bound}
In addition to the assumptions of Theorem \ref{thm:ulaq1}, suppose the assumptions \textbf{(a.9)} and \textbf{(a.10)} hold. Then, \textbf{(U.2)}-\textbf{(U.3)} hold true.
\end{lem}
\begin{proof}
The assumption \tbf{(a.6)} implies that $\mE|\cL(\B_{0})|<\iny$, and hence, \textbf{(U.2)} will follow from Chevyshev's inequality. Next, let $\bu:=\bA^{-1}(\B-\B_{0})\in\mR^{p}$. As done in Lemma 5.5.4 from \cite{Koul2002}, define $V_{j}(\bu)=\inhy{\cW_{j}(y, \B_{0}+\bA\bu)l(y)}$ and $\wh{V}_{j}(\bu)=\inhy{\{\cW_{j}(y, \B_{0})+\bu'\bA\bR_{j}(y,\B_{0})\}l(y)}$ for $1\le j\le p$. Subsequently, define two vectors $\mbf{V}(\bu):=(V_{1},...,V_{p})'\in\mR^{p}$ and $\wh{\mbf{V}}(\bu):=(\wh{V}_{1},...,\wh{V}_{p})'\in\mR^{p}$. Let $\bu=r\mbf{e}$ where $r>0$ and $\mbf{e}\in\mR^{p}$ is a unit vector. Using these variables, it can be shown that both $\mbf{e}\mbf{V}(\bu)'$ and $\mbf{e}\wh{\mbf{V}}(\bu)'$ are monotone in $\|\bu\|$ under \textbf{(a.9)} and \textbf{(a.10)}. Also, for $b<\iny$, the analogue of (5.5.26) from \cite{Koul2002}
\benn
\sup_{\|\bu\|\le b}\|\mbf{V}(\bu)-\wh{\mbf{V}}(\bu) \|=o_{p}(1)
\eenn
will follow from Lemma \ref{lem:mean_value_theorem_F}, $\inhy{l^2(y)}<\iny$ in \tbf{(a.9)}, and the Cauchy-Schwarz inequality. Then, using the monotonicity of $\mbf{e}\mbf{V}(\bu)'$ and $\mbf{e}\wh{\mbf{V}}(\bu)'$, the claim for \tbf{(U.3)} can be shown as in \cite{Koul2002}.
\end{proof}
\noi
Ascertaining that the ULAQ conditions are met, we proceed to prove the asymptotic normality of the MD estimator. We first specify the asymptotic distribution of $\bSn$ and convergence of $\bW_{n}$ in Lemma \ref{lem:koul_asym_distr_Sn}. Then, the subsequent application of Lemma \ref{lem:koul_asym_distr_Sn} will yield the desired result, which is another main result of this paper; see, e.g., Theorem \ref{thm:asymp_normality}.
Let $\wt{\G}_{n}(\B):=\inhy{\G_{n}(\B;y)'\G_{n}(\B;y)}$.
\begin{lem}\label{lem:asymp_distr_S}
Assume \tbf{(a.1)}-\tbf{(a.10)}. In addition, assume that
\benn
\lim_{n\ra \iny}\wt{\G}_{n}=\wt{\G},
\eenn
where $\wt{\G}$ is positive-definite.
Then $\bA\bSn(\B_{0})$ is asymptotically normally distributed, and $\bA\bW_{n}\bA$ converges to $\wt{\G}$ as $n$ approaches $\iny$.
\end{lem}
\noi
\begin{proof}
The claim for $\bW_{n}$ readily immediately follows from (\ref{eq:SandW_Matrix}) and the assumption. Note that for a real-valued function $\ell:\mR\ra\mR$,
\benn
\inhy{\ell(y)\Big\{I(Y_{i}\le y)-F_{i}(y;\B_{0})\Big\}} = \int_{Y_{i}}^{\iny}\ell(y)dH(y)-\mE\left(\int_{Y_{i}}^{\iny}\ell(y)dH(y)\right).
\eenn
Recall the $j$th column vector of $\G_{n}$: $\mbf{\gamma}_{j}=(\gamma_{1j},...,\gamma_{pj})'\in\mR^{p}$. For any $\mbf{b}=(b_{1},...,b_{p})'\in \mR^{p}$, we have
\benrr
\mbf{b}'\bA\bSn(\B_{0}) &=& \sti\stj d_{ij} \sum_{l=1}^{p}b_{l}\inhy{\gamma_{lj}(y,\B_{0})\Big\{I(Y_{i}\le y)-F_{i}(y;\B_{0})\Big\}},\\
&=& \sti\stj d_{ij} \sum_{l=1}^{p}b_{l}\Big( \psi_{lj}(Y_{i}) - \mE\{\psi_{lj}(Y_{i})\} \Big),\\
&=& \sti \xi_{i},\qquad (say),
\eenrr
where $\psi_{lj}(Y_{i}):=\int_{Y_{i}}^{\iny}\gamma_{lj}(y,\B_{0})dH(y)$. Note that \tbf{(a.8)} implies $|\psi_{lj}(Y_{i})|<\iny$, which implies $|\xi_{i}|$ is bounded by $c\,\max |d_{ij}|$ for some constant $c<\iny$. We shall show that the Lindeberg-Feller (L-F) condition for $\mbf{b}'\bA\bSn$ will be satisfied. $\mE{\xi_{i}}=0$ is clear. Let $\si_{i}^{2}=Var(\xi_{i})$ and $\tau_{n}=\sti \si_{i}^{2}$. Thus, we have for any $\epsilon>0$,
\benrr
\frac{1}{\tau_{n}^{2}}\sti \mE[\xi_{i}^{2}:|\xi_{i}|>\epsilon \tau_{n}]&\le & c\tau_{n}^{-2}\left(\max_{1\le i\le n,1\le j\le p}d_{ij}^{2}\right)\sti \mP(|\xi_{i}|>\epsilon \tau_{n}) \\
&\le & c\epsilon^{-2}\tau_{n}^{-2}\left(\max_{1\le i\le n,1\le j\le p}d_{ij}^{2}\right) \ra 0,
\eenrr
where the second one is immediate after application of the Chevyshev inequality to the summand, while the convergence to zero follows from \tbf{(a.2)}, thereby showing that the L-F condition is met. Define an $n\times n$ diagonal matrix $\L_{n}$ whose $i$th entry is $F_{i}(y)(1-F_{i}(y))$. Note that
\benrr
Var(\mbf{b}'\bA\bSn)&=&\mbf{b}'\mE\inhy{\G_{n}'\W\W'\G_{n}} \mbf{b},\\
&=&\mbf{b}'\inhy{\G_{n}'\bD'\bL_{n}\bD\G_{n}} \mbf{b},\\
&=&\mbf{b}'\bO_{n} \mbf{b},\quad (say),
\eenrr
where the second last equality follows from the Fubini's theorem, and $\mE(\W\W')=\bD'\bL_{n}\bD$. Thus, the Cramer-Wold device will yield the asymptotic normality of $\bA\bSn$
\benn
\bO_{n}^{-1/2}\bA\bSn\Rightarrow_{\cD}N(\mbf{0}_{p\times 1}, \mathbf{I}_{p\times p}).
\eenn

\end{proof}
\noi
Finally, we conclude this section by stating the  asymptotic normality of the MD estimator. Recall $\wt{\G}_{n}(\B_{0})$ and $\bO_{n}(\B_{0})$ in Lemma \ref{lem:asymp_distr_S}. Define $\S_{n}(\B_{0}) := \wt{\G}_{n}^{-1}\bO_{n}\wt{\G}_{n}^{-1}$.
\begin{thm}\label{thm:asymp_normality}
Suppose the assumptions of Lemma \ref{lem:asymp_distr_S} hold. Then the MD estimator $\widehat{\B}$ in (\ref{eq:opt}) will be asymptotically normally distributed
\benn
\S_{n}^{-1/2}\bA^{-1}(\wh{\B}-\B_{0})\Rightarrow_{\cD}N(\mbf{0}_{p\times 1}, \mathbf{I}_{p\times p}).
\eenn
\end{thm}
\begin{proof} Note that the ULAQ conditions are met by
Theorem \ref{thm:ulaq1} and Lemma \r{lem:asym_bound}, and hence, Lemma \ref{lem:koul_asym_distr_Sn} accompanied by Lemma \ref{lem:asymp_distr_S} will immediately imply the claim, thereby completing the proof of the theorem.
\end{proof}

\subsection{Asymptotic variance of the MD estimator}
We will find the asymptotic variance of $\wh{\B}$ when the integrating measure in the distance function is the Lebesgue measure, that is, $H(y)\equiv y$. To begin with, we will find $\wt{\G}_{n}$. Let $\mathbf{C}_{H}:=\int \bG\bD^{*}\bG\,dH(y)$ and hence,
\benn
\wt{\G}_{n}(\B) = \bA\bX'\left(\int \bG(y,\B)\bD^{*}\bG(y,\B)dH(y) \right) \bX\bA = \bA\bX'\mathbf{C}_{H}\bX\bA.
\eenn
Let $c_{ij}$ denote the $(i,j)$th entry of $\mathbf{C}_{H}$. Note that
\ben\label{eq:c_ij}
c_{ij} =
d_{ij}^{*}\int g_{i}(y,\B) g_{j}(y,\B)dH(y) = \left\{
  \begin{array}{ll}
    \frac{2d_{ij}^{*}\la_{i}\la_{j}}{(\la_{i}+\la_{j})^3}, & \hbox{if $H(y)\equiv y$;} \\
    \frac{2d_{ij}^{*}\la_{*}\la_{i}\la_{j}}{(\la_{*}+\la_{i}+\la_{j})^3}, & \hbox{if $H(y)=1-e^{-\la_{*}y}$.}
  \end{array}
\right.
\een
Next, proceed to get $Var(\bSn) = \mE(\bSn\bSn')$. Recall the diagonal matrix $\bP_{n}$. Let $\bP_{H}:=\int \bG\bD^{*}\bP_{n} \bD^{*}\bG\,dH(y)$ and $p_{ij}$ denote its $(i,j)$th entry. Observe that
\benrr
p_{ij}(\B) &=& \sum_{k=1}^{n}d_{ik}^{*}d_{kj}^{*}\la_{i}\la_{j}\int y^2e^{-(\la_{i}+\la_{j})y}F_{k}(y;\B)\{1-F_{k}(y;\B)\}dH(y)\\
&=& \left\{
\begin{array}{ll}
    \sum_{k=1}^{n}d_{ik}^{*}d_{kj}^{*}\la_{i}\la_{j}\left(\frac{2}{(\la_{i}+\la_{j}+\la_{k})^3}-\frac{2}{(\la_{i}+\la_{j}+2\la_{k})^3}\right), & \hbox{if $H(y)\equiv y$;} \\
   \sum_{k=1}^{n}d_{ik}^{*}d_{kj}^{*}\la_{*}\la_{i}\la_{j}\left(\frac{2}{(\la_{*}+\la_{i}+\la_{j}+\la_{k})^3}-\frac{2}{(\la_{*}+\la_{i}+\la_{j}+2\la_{k})^3}\right), & \hbox{if $H(y)=1-e^{-\la_{*}y}$.}
  \end{array}
\right.
\eenrr
Let $\eta_{i}(y,\B):=\textrm{I}(Y_{i}\le y)-F_{i}(y;\B)$ and $\bet(y,\B)=(\eta_{1}(y,\B),...,\eta_{n}(y,\B))'\in \mR^{n}$. Observed that $\W(y,\B)=\bD'\bet$, and  $\bSn$ can be expressed in a matrix form, namely, $\bSn=-\bX'\int \bG\bD^{*}\bet(y,\B)dy$ where $\bD^{*}=\bD\bD'$. Finally, we have
\benrr
\bO_{n} &=& \mE(\bA\bSn\bSn'\bA),\\
&=& \bA\bX'\left(\int \bG(y,\B)\bD^{*}\mE\left[\bet\bet'\right] \bD^{*}\bG dy\right)\bX\bA,\\
&=& \bA\bX'\left(\int \bG(y,\B)\bD^{*}\bP_{n} \bD^{*}\bG dy\right)\bX\bA,\\
&=& \bA\bX'\bP_{H}\bX\bA.
\eenrr
Thus, the aymptotic variance of the MD estimator $AVar(\wh{\B})$ can be written as
\benn
AVar(\wh{\B}) = (\bX'\mathbf{C}_{H}\bX)^{-1}(\bX'\bP_{H}\bX)(\bX'\mathbf{C}_{H}\bX)^{-1}.
\eenn

\subsection{Robustness of the MD estimator}
Next, we will investigate the robustness of $\wh{\B}$. To that end, we will first find its influence function. Note that $\bSn(\B)$ can be rewritten as
\benrr
\bSn(\B) &=& \sti \inhy{\cW_{i}^{*}(y,\B) \bqi(y,\B)},\\
&=& \sti \psi(Y_{i}, \bxi, \B), \quad (say),
\eenrr
where $\cW_{i}^{*}(y,\B):=\stj d_{ij}\cW_{j}(y,\B)$. Since $\partial \cL/\partial \B=-2\bSn$, the MD estimator $\wh{\B}$ will solve
\benn
\sti \psi(Y_{i}, \bxi, \B)=0.
\eenn
Let $\textrm{IF}(Y_{i}, \bxi, \B)$ denote the influence function of the MD estimator. We shall compute the influence function of the MD estimator by directly applying the formula from \citet[p.\,101]{Hampel1986} to $\psi$ as follows:
\benn
\textrm{IF}(Y_{i},\bxi;\B) = -\left(\mE\left[\frac{\partial \psi(Y_{i}, \bxi;\B)}{\partial \B}\right] \right)^{-1} \psi(Y_{i}, \bxi;\B).
\eenn
Since $\mE \cW_{i}^{*}=0$ and $\partial \cW_{j}/\partial \B=\sum_{h=1}^{n}g_{h}(y,\B)\mbf{x}_{h}$, the application Fubini's theorem yields
\benn
\mE\left[\frac{\partial \psi(Y_{i}, \bxi;\B)}{\partial \B}\right] = -\sum_{h=1}^{n}d_{hi}^{*} \bxi\mbf{x}_{h}' \inhy{g_{i}(y,\B)g_{h}(y,\B)},
\eenn
where $d_{hi}^{*}=\stj d_{ij}d_{hj}$. Consider $H(y)\equiv y$. Then, (\ref{eq:c_ij}) directly implies
\benn
\mE\left[\frac{\partial \psi(Y_{i}, \bxi;\B)}{\partial \B}\right] = -\frac{1}{2}\sum_{h=1}^{n}c_{ih} \bxi\mbf{x}_{h}'
\eenn
Observe that
\benrr
\psi(Y_{i}, \bxi, \B)&=&\bxi \sum_{h=1}^{n}d_{ih}^{*}\int g_{i}(y,\B)\left[\textrm{I}(Y_{i}\le y)-F_{i}(y;\B)\right]\,dH(y),\\
&=& \bxi\xi_{i}(\B),\quad (say),
\eenrr
hence, the influence function can be written as
\benn
\textrm{IF}(Y_{i},\bxi;\B) = -2\left(\sum_{h=1}^{n}c_{ih} \bxi\mbf{x}_{h}'\right)^{-1} \xi_{i}(\B)\bxi.
\eenn
For $H(y)\equiv y$ or $H(y)=1-e^{\la_{*}y}$, $\xi_{i}(\B)$ has an closed-form expression: see, e.g., (\ref{eq:xi_k}). With $H(y)\equiv y$, consider a scalar outlier $x_{i}\in \mR$. When $x_{i}$ approaches $\iny$, which implies $\la_{i}$ approaches $\iny$, we have
\benrr
\textrm{IF}(Y_{i},x_{i};\beta)
&=& -\frac{\sum_{h=1}^{n}d_{ih}^{*}[Y_{h}e^{-\la_{i}Y_{h}}/\la_{i}+e^{-\la_{i}Y_{h}}/\la_{i}^2-\la_{h}(2\la_{i}+\la_{h})/\la_{i}^{2}(\la_{i}+\la_{h})^2]}{\sum_{h=1}^{n}d_{ih}^{*}x_{h}\la_{h}/(\la_{i}+\la_{h})^3},\\
&\approx&-\frac{2\sum_{h\neq i}^{n}d_{ih}^{*}\la_{h}+2d_{ii}^{*}\la_{i}}{\sum_{h\neq i}^{n}d_{ih}^{*}x_{h}\la_{h}+d_{ii}^{*}x_{i}\la_{i}}\rightarrow 0,
\eenrr
where the equality follows from (\ref{eq:xi_k}) while the approximation of the second line follows from the fact that the maximum power of $\la_{i}$ in both numerator and denominator is 3 and $x^{k}e^{-x}$ converges to 0 for all $k\in \mathbb{N}$. Thus, the impact of any outlier on the MD estimation will be limited as shown above. For $H(y)=1-e^{-\la_{*}y}$, the similar result holds. To demonstrate that the assertion is true, we investigate an empirical influence function obtained from randomly generated dataset $(Y_{i}, x_{i}),\,1\le i\le 100$ with true $\beta=0.2$. Letting one of $x_{i}$ increase, we compute $\textrm{IF}(Y_{i},x_{i};\beta)$ using $H(y)\equiv y$ and $H(y)=1-e^{-y}$.
\begin{figure}[h]
\centering
\includegraphics[width=0.5\textwidth]{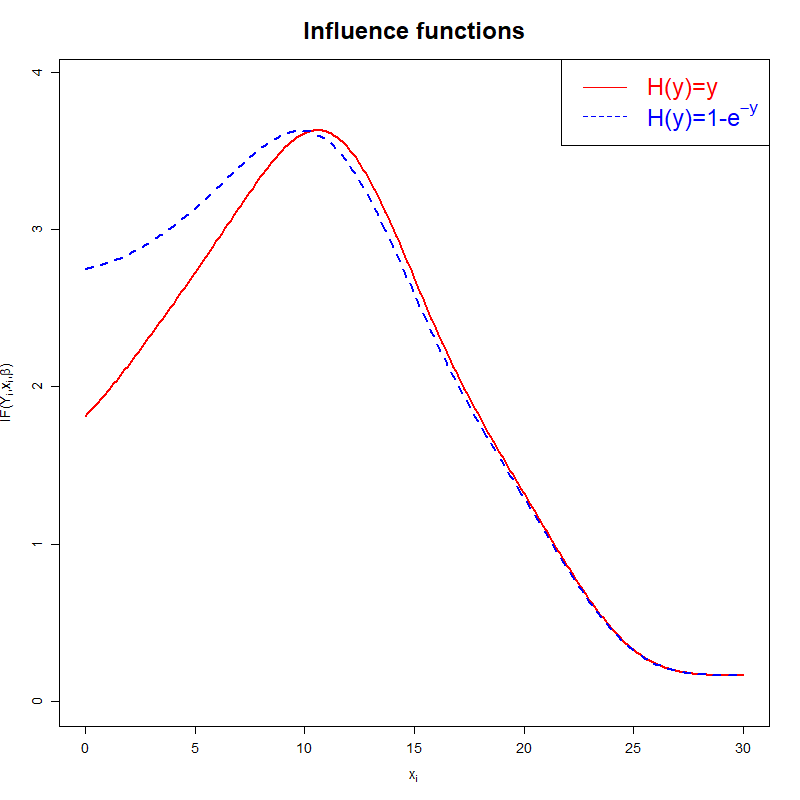}
\caption{Two influence functions when $H(y)=y$ (red) and $H(y)=1-e^{y}$ (blue)}\label{fig:influence_function}
\end{figure}
Figure \ref{fig:influence_function} shows two influence functions corresponding to different integrating measures and serves to illustrate that the impact of the outlier is indeed limited. Note that only significant difference between the two measures arises when $x_{i}$ is relatively small; as $x_{i}$ increases, both influence functions converge to 0, thereby demonstrating the robustness of the MD estimator.

\section{Empirical studies}
\subsection{Computational aspects of MD estimation}
As stated in Section \ref{Sec:CvM_distance_function}, $\cL$ will be a quadratic function as $n$ increases. We will empirically demonstrate that the claim indeed holds true. Recall $\W(y,\B) = \bD' \bet$. With $H(y)\equiv y$, $\cL(\B)$ can be rewritten as
\benrr
\cL(\B) &=& \int \W'(y,\B)\W(y,\B) dy,\\
&=& \stk \sti d_{ki}^{*}\int\eta_{k}(y,\B)\eta_{i}(y,\B)\,dH(y),\\
&=& \stk \sti d_{ki}^{*}\left\{ (Y_{k}\wedge Y_{i})+\frac{e^{-\la_{k}Y_{i}}-1}{\la_{k}}+\frac{e^{-\la_{i}Y_{k}}-1}{\la_{i}} +\frac{1}{(\la_{k}+\la_{i})}\right\},
\eenrr
where $\la_{i}=\bxi\B, 1\le i\le n$ and $\la_{k}=\mbf{x}_{k}\B, 1\le k\le n$. Now, consider the case that $H(y)$ is an exponential df with a rate parameter $\la_{*}$, namely, $H(y)=1-e^{\la_{*}y}$. Then, $\cL(\B)$ has a similar -- albeit more complicated -- expression where
\benrr
\int\eta_{k}(y,\B)\eta_{i}(y,\B)\,dH(y)&=&\frac{\la_{*}}{\la_{*}+\la_{k}}e^{-(\la_{*}+\la_{k})Y_{i}}+\frac{\la_{*}}{\la_{*}+\la_{i}}e^{-(\la_{*}+\la_{i})Y_{k}}-\left(\frac{\la_{*}}{\la_{*}+\la_{k}}+\frac{\la_{*}}{\la_{*}+\la_{i}}\right)\\
&&\quad + \frac{\la_{*}}{\la_{*}+\la_{k}+\la_{i}}-e^{-\la_{*}(Y_{k}\wedge Y_{i})}+1.
\eenrr
Using the above expressions, we will plot the 3 dimensional graph of $\cL(\B)$. More precisely, generate $\bxi\in\mR^{2}$ from a uniform distribution and obtain $\la_{i}$, using $\B_{0}=(-2,3)$. Then generate $Y_{i},\,1\le i\le n$ from the exponential distribution with the rate $\la_{i},\,1\le i\le n$. Using this dataset $(Y_{i},\bxi)$, we will draw $\cL(\B)$ over the neighborhood of $\B_{0}$. Figure \ref{fig:loss} demonstrates $\cL$ is indeed quadratic around the true $\B_{0}$.
\begin{figure}[h]
\centering
\begin{subfigure}[b]{0.45\textwidth}
\includegraphics[width=\textwidth]{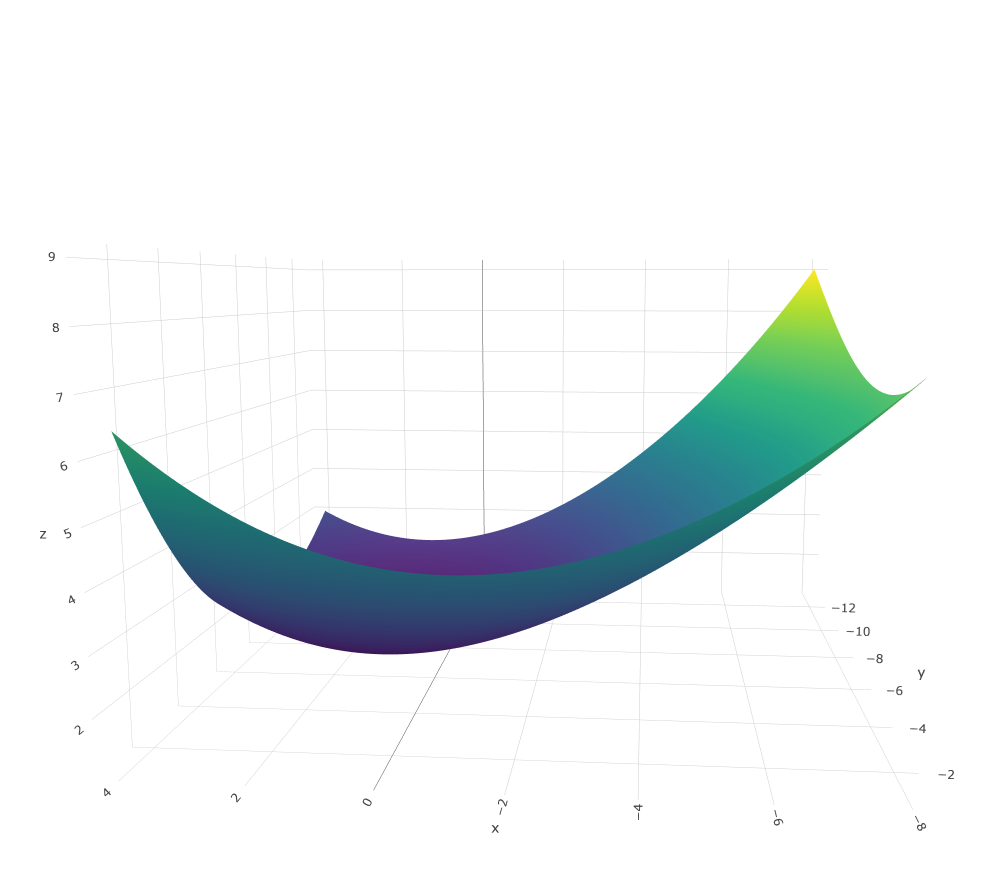}
\end{subfigure}
\begin{subfigure}[b]{0.45\textwidth}
\includegraphics[width=\textwidth]{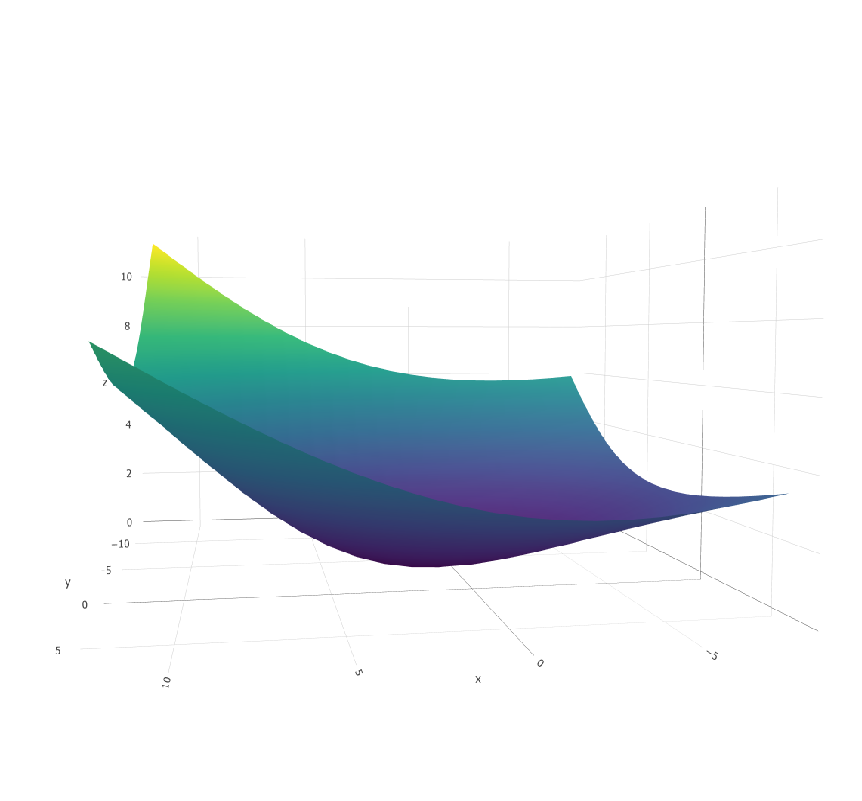}
\end{subfigure}
\caption{3 dimensional plots of $\cL(\B)$ when $H(y)=y$ (left) and $H(y)=1-e^{y}$ (right).}\label{fig:loss}
\end{figure}

Ascertaining the quadraticity of $\cL$, we can safely use the gradient descent (GD) method when computing the MD estimator. Recall $\partial \cL/\partial \B=2\bSn$. Thus, the application of GD method implies
\benn
\wh{\B}^{(i+1)}=\wh{\B}^{(i)}-\Delta\cdot \bSn,
\eenn
where $\wh{\B}^{(i)}$ is the MD estimator in the $i$th stage while $\Delta$ is a learning rate. For the learning rate, we will use $\Delta=0.001$; for the initial value $\wh{\B}^{(i)}$, we will use a vector of zeros, namely, $\mbf{0}\in \mR^{p}$. We are fortunate to see that $\bSn$ has also a closed-form expression, and hence, the MD estimator can be obtained pretty much fast. Recall that $d_{ij}^{*}, 1\le, i,j\le n$ is the $(i,j)$th entry of $\bD^{*}$. Let $\bbxi(\B):= \int \bG\bD^{*}\bet(y,\B)dH(y)$ and $\xi_{k}$ denote its $k$th entry. Note that
\ben\label{eq:xi_k}
\xi_{k}(\B) = \sti d_{ki}^{*}\int g_{k}(y,\B)\eta_{i}(y,\B)dy = \sti d_{ki}^{*}\left[\phi_{k}(Y_{i})-\mE\,\phi_{k}(Y_{i}) \right],
\een
where $\phi_{k}(x):=\int_{x}^{\iny}g_{k}(y, \B)dH(y)$. When $H(y)\equiv y$, a direct calculation yields
\benn
\phi_{k}(x)= xe^{-\la_{k}x}+\frac{1}{\la_{k}}e^{-\la_{k}x},\quad \mE[\phi_{k}(Y_{i})] = \frac{\la_{i}(2\la_{k}+\la_{i})}{\la_{k}(\la_{k}+\la_{i})^2},
\eenn
while using $H(y)=1-e^{-\la_{*}y}$ will yield
\benn
\phi_{k}(x)= \frac{\la_{*}\la_{k}}{(\la_{*}+\la_{k})}xe^{-(\la_{*}+\la_{k})x}+\frac{\la_{*}\la_{k}}{(\la_{*}+\la_{k})^2}e^{-(\la_{*}+\la_{k})x},\quad \mE[\phi_{k}(Y_{i})] = \frac{\la_{*}\la_{k}\la_{i}(2\la_{*}+2\la_{k}+\la_{i})}{(\la_{*}+\la_{k})^2(\la_{*}+\la_{k}+\la_{i})^2}.
\eenn
Next consider a mixture of the two previous measures, namely, $H(y):=ry+(1-r)(1-e^{-\la_{*}y})$ with $0<r<1$. Then, $\phi_{k}(x)$ and $\mE[\phi_{k}(Y_{i})]$ will be linear combinations of those corresponding to the original measures. Let $\boldsymbol{\Phi}$ be an $n\times n$ matrix whose $(i,k)$th entry is $\Phi_{ik} = \phi_{k}(Y_{i})-\mE \phi_{k}(Y_{i})$. Then, $\xi_{k}(\B)$ will be the $(k,k)$th entry of the  matrix $\bD^{*}\boldsymbol{\Phi}$. Note that $\bSn=-\bX'\bbxi$. Thus, we first compute $\bD^{*}$ and $\boldsymbol{\Phi}$, and subsequently obtain $\bbxi$; finally, we obtain $\bSn$ by multiplying $\bbxi$ by $-\bX$. Then, we apply the GD method using $\bSn$ and obtain the MD estimator. We will refer to the MD estimators obtained from using $H(y)\equiv y$, $H(y)=1-e^{\la_{*}y}$, and the mixture of the two measures as MDE1, MDE2, and MDE3, respectively. For MDE3, $r=0.5$ will be used in what follows, unless specified otherwise.

\subsection{Simulation studies}
Using the generated dataset, we compare various MD estimators -- obtained from different integrating measure introduced in the previous sections -- with the Cox estimator. Recall the MD1 and MD2 estimators obtained from $H(y)\equiv y$ and $H(y)=1-e^{-y}$, respectively. In addition, we use a mixture of the two $H$'s -- e.g., $H(y):=0.5 y +0.5(1-e^{-y})$ -- and obtain another MD estimator; we refer to it as the MD3 estimator.

For the dataset, we use $\B=(2,-3)'$ again and generate $\bxi\in \mR^2,\,1\le i\le n$ from a normal distribution with mean of 1 and standard deviation of 0.1, where $n$ is 20, 50, 100, or 200. Finally, we generate $Y_{i}, 1\le i\le n$ using the rate $\la_{i}=\bxi'\B$. With the generated dataset, we compute three MD estimators -- MD1, MD2, and MD3 -- and the Cox estimator; for getting the Cox estimator, we use the R package \texttt{survival}. Finally, we repeat the whole procedure 10,000 times.
For comparison purpose, we use three criteria: bias, standard error (SE), and root mean square error (RMSE).
Table \ref{tbl:comparison} reports the comparison of the four estimators.
\begin{table}[h]
\centering
\small{
\begin{tabular}{|c|c|c c c|c c c|c c c|c c c|}
  \hline
 &  & \multicolumn{3}{|c|}{MD1} & \multicolumn{3}{|c|}{MD2} & \multicolumn{3}{|c|}{MD3} & \multicolumn{3}{|c|}{Cox}\\
\hline
 &$n$ & bias & SE & RMSE & bias & SE & RMSE  & bias & SE & RMSE & bias & SE & RMSE \\
\hline
  \multirow{4}{*}{$\beta_{1}$} &20  &-0.9&1.465&1.719   &-1.561&0.643&1.688 &-1.23&1.034&1.607  &0.11&3.255&3.255\\
&50  &-0.221&1.307&1.324 &-1.014&0.778&1.278 &-0.617&1.018&1.19  &0.182&1.618&1.627\\
&100 &-0.063&1.104&1.105 &-0.572&0.868&1.039 &-0.318&0.962&1.013 &0.023&1.118&1.118\\
&200 &-0.018&0.729&0.729 &-0.168&0.729&0.748 &-0.093&0.711&0.717 &0.009&0.729&0.728\\
  \hline
  \multirow{4}{*}{$\beta_{2}$} &20  &1.209&1.491&1.919  &2.27&0.651&2.361  &1.739&1.048&2.031 &-0.388&3.216&3.238\\
&50  &0.395&1.37&1.426   &1.539&0.831&1.749 &0.967&1.073&1.444 &-0.12&1.712&1.715\\
&100 &-0.003&1.089&1.089 &0.771&0.847&1.145 &0.384&0.94&1.015  &-0.101&1.092&1.096\\
&200 &-0.017&0.802&0.802 &0.225&0.777&0.809 &0.104&0.768&0.774 &-0.065&0.764&0.767\\
  \hline
\end{tabular}
}
\caption{Biases and SEs of the Cox and MD estimators. }\label{tbl:comparison} 
\end{table}
We first interpret the general trends across all estimators and proceed to the estimator-specific analysis. To begin with, it is worth noting that all estimators exhibit the consistency, that is, absolute value of bias tends to decrease as $n$ increases. To reflect the sample size effect, other two measures (SE and RMSE) of all estimators also consistently decrease.

Next, we shall analyze the results separately by estimation. Note that MD1 performs well for relatively larger $n$ in terms of bias. Regarding $\b_{1}$, at $n=$100 and 200, MD1 exhibits the second lowest bias (-0.063 and -0.018, respectively), following Cox's bias (0.023 and 0.009, respectively). For $\b_{2}$, it shows highly accurate bias -- -0.003 and -0.017 at $n=100$ and 200, respectively -- and hence reports the smallest bias among all estimators. MD2 yields the smallest SE across all $n$'s except $n=200$ for both $\b_{1}$ and $\b_{2}$. However, it suffers from the largest bias among all estimators, which indicates the bias-variance tradeoff. Next, consider $MD3$. As expected from the fact that a mixture of two integrating measures of MD1 and MD2 is used to obtain the MD3 estimator, it is not unreasonable to guess that it acts a middle ground between two MD estimator, which turns out to be true. More specifically, its bias, SE, and RMSE typically sit between MD1 and MD2 across all $n$'s for both $\b_{1}$ and $\b_{2}$. Therefore, MD3 shows the balanced performance between MD1 and MD2. Finally, consider the Cox estimator. Regarding the estimation of $\b_{1}$, it outperforms all MD estimators in terms of bias, whereas it suffers from the largest SE for all $n$'s; as a result, Cox yields the largest RMSE for almost all $n$'s, which is another evidence of the bias-variance tradeoff. Regarding the estimation of $\b_{2}$, Cox exhibits the similar pattern; it yields smaller bias, larger SE, and hence, larger RMSE than other MD estimators.
\begin{figure}[h]
\centering
\begin{subfigure}[b]{0.45\textwidth}
\includegraphics[width=\textwidth]{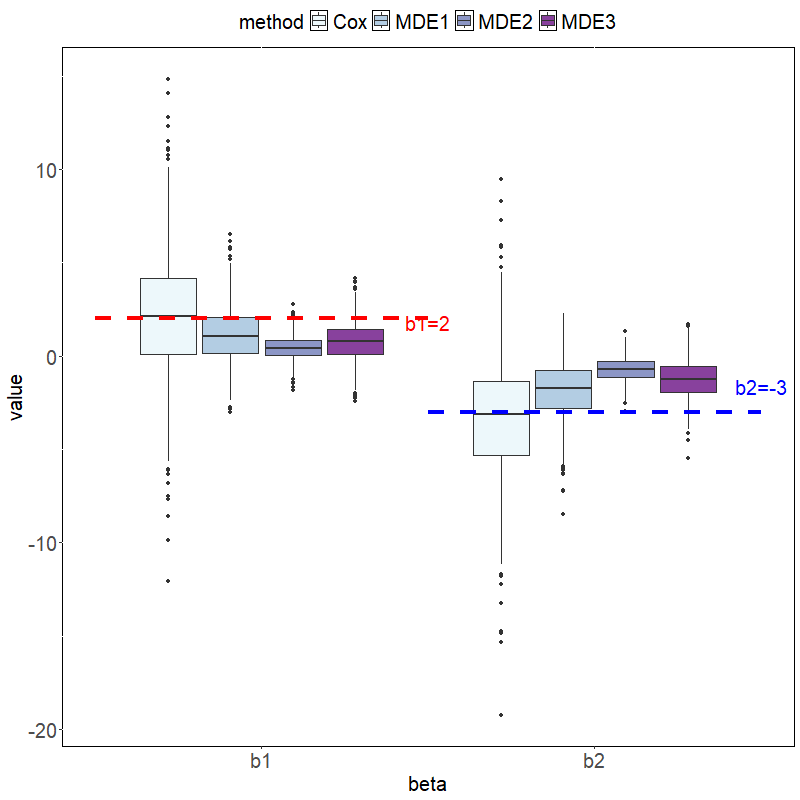}
\end{subfigure}
\begin{subfigure}[b]{0.45\textwidth}
\includegraphics[width=\textwidth]{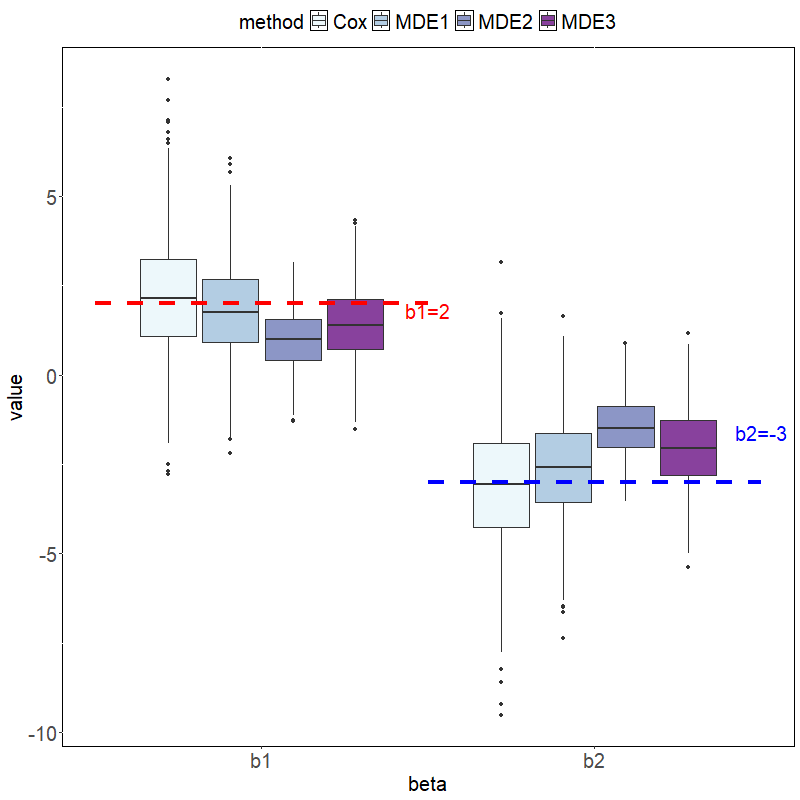}
\end{subfigure}
~
\begin{subfigure}[b]{0.45\textwidth}
\includegraphics[width=\textwidth]{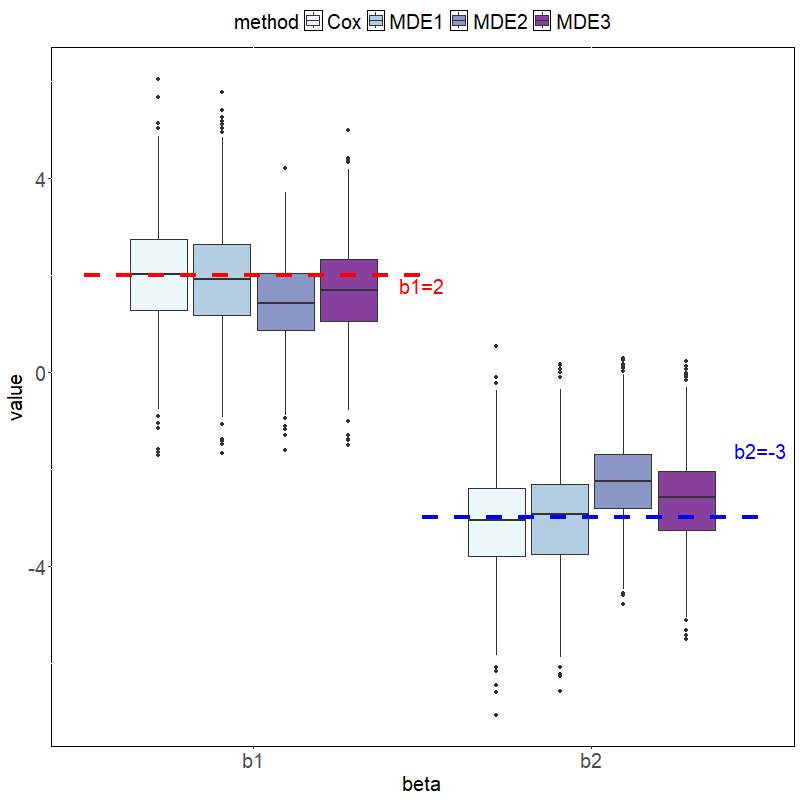}
\end{subfigure}
\begin{subfigure}[b]{0.45\textwidth}
\includegraphics[width=\textwidth]{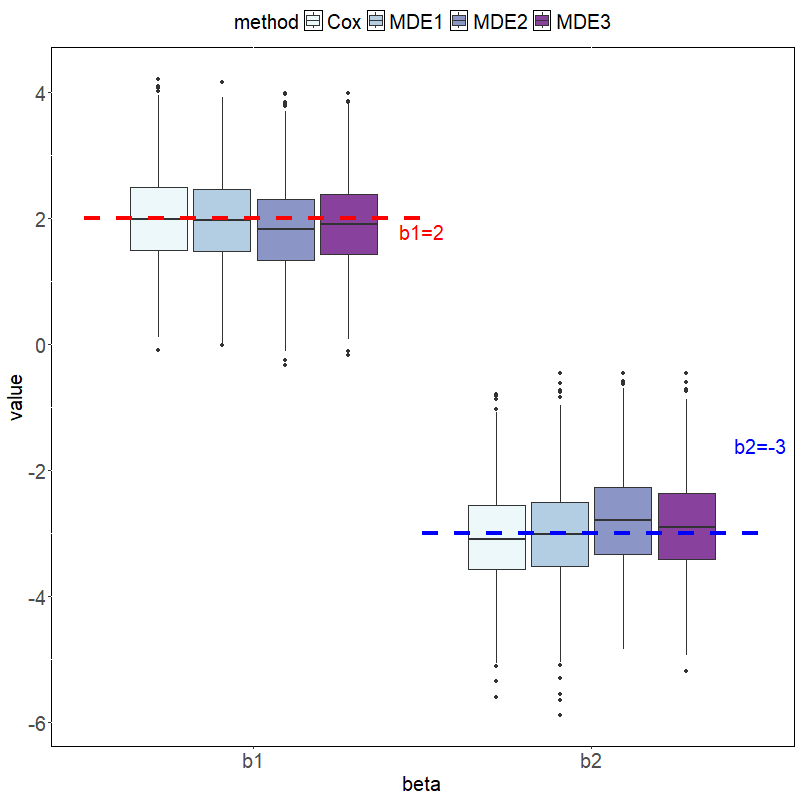}
\end{subfigure}
\caption{Boxplot of estimators when $n=20$ (top-left), 50 (top-right), 100 (bottom-left), and 200 (bottom-right).}\label{fig:boxplot}
\end{figure}
Figure \ref{fig:boxplot} reports boxplots of the four estimators obtained from 10,000 iterations. The boxplots visualize the results of the previous simulation experiment and provide more comprehensive and straightforward interpretation. Note that the results reported in the figure closely accords with those reported in Table \ref{tbl:comparison}; all features reported in the table are also illustrated in the figure, such as  the consistency of all estimators, smaller bias but larger SE of the Cox estimator, the opposite measures of MD estimators, and the balanced performance of MD3 between MD1 and MD2.

\section{Conclusion}\label{Sec:conclusion}
This study applied the MD estimation with the CvM-type distance function, along with a different approach proposed by \cite{Kim2026}, to estimate the parameter of the regression setup of the exponential distribution and demonstrated that the MD estimator still retains desirable properties, such as asymptotic normality and robustness.
The constant rate over time, which is a bit strong and rare assumption in real-world application, is the definite limit of this study. Thus, further extensions of the current study by releasing the constant rate assumption will form the basis for future research.

\bibliographystyle{plainnat}
\bibliography{MDE_ref.bib}

\edt

felicitous, fetching

\newpage
\section{Influence curve}
\begin{lem}
Recall $W(p)$ in Lemma \r{lem:SW}. The influence function of the MD estimator $\hat{p}$ is given by
\benn
IF(x,\hat{p}, F)= 2W(p)^{-1}\sum_{k=0}^{m}m_{k}g_{k,m}(1-p)\{I(x\leq k) - F(k;m,p)\}.
\eenn
\end{lem}
\noi
\textbf{Proof.} From $(\r{eq:rewritten_loss})$, the MD estimator $\hat{p}$ solves
\benn
\sum_{k=0}^{m}\cW_{d}(k, p)\,\, \frac{\partial \cW_{d}(k, p)}{\partial p} =0.
\eenn
Using (\r{eq:modelled_F}), we have
\benn
\frac{\partial}{\partial p}\cW_{d}(k, p) = \sqrt{n}m_{k}g_{m,k}(1-p).
\eenn
Hence, the MD estimator $\hat{p}$ solves
\benn
\sti \phi(X_{i}, p) = 0,
\eenn
where
\benn
\phi(X_{i}, p) = \sum_{k=0}^{m}m_{k}g_{m,k}(1-p)\Big\{I(X_{i}\leq k) - F(k;m,p)\Big\}.
\eenn
For the contaminated distribution $F_{\eps}:=(1-\eps)F+\eps\delta_{x}$, $\hat{p}$ solves
\benn
(1-\eps)E_{F}[\phi(X,p)] + \eps\phi(x,p) = 0.
\eenn
Using $E_{F}[\phi(X,p)]=0$ around $\eps=0$, the differentiation with respect to $\eps$ yields
\benn
\frac{\partial p}{\partial \eps}=
\left[E_{F}\left(\frac{\partial}{\partial p}\phi(X,p)\right)\right]^{-1}\phi(x,p).
\eenn
Note that
\benn
E_{F}\left(\frac{\partial}{\partial p}\phi(X,p)\right) = \sum_{k=0}^{m}m_{k}^{2}g_{m,k}^{2}(1-p)=W(p)/2,
\eenn
thereby completing the proof of the lemma.


\bibliographystyle{plain}
\bibliography{MDE_ref.bib}

@article{Beran,
  author		= {Beran, R. J.},
  title			= {Minimum Helinger distance estimates for parameter models},
  journal		= {Ann. Statisti.},
  volume		= {5},
  number={},
  pages			= {445--463},
  year			= {1977}
}

@article{ Donoho88a,
  author		= {Donoho, D. L. and Liu, R. C.},
  title			= {The Automatic robustness of minimum distance functionals},
  journal		= {Ann. Stat.},
  volume		= {16},
  number = {2},
  pages			= {552--586},
  year			= {1988a},
  mrnumber   = {0947562}
}

@article{ Donoho88b,
  author		= {Donoho, D. L. and Liu, R. C.},
  title			= {Pathologies of some minimum distance estimators},
  journal		= {Ann. Stat.},
  volume		= {16},
  number={2},
  pages			= {587--608},
  year			= {1988b},
  mrnumber   = {0947563}
}

@book{Hampel1986,
  author		= {Hampel, F. R. and Ronchetti, E. M. and Rousseeuw, P. J. and Stahel, W. A.},
  title			= {Robust statistics: the approach based on influence functions},
  address		= {New York},
  publisher		= {Wiley},
  year			= {1986}

}

@article{ Kim2018,
  author		= {Kim, J.},
  title			= {A fast algorithm for the coordinate-wise minimum distance estimation},
  journal		= {Comput. Stat.},
  volume		= {88},
  number={3},
  pages			= {482--497},
  year			= {2018},
  mrnumber   = {3764832}
}

@article{Kim2020,
  author		= {Kim, J.},
  title			= {Minimum distance estimation in linear regression with strong mixing errors},
  journal		= {Commun. Stat.-Theory Methods.},
  volume		= {49},
  number={6},
  pages			= {1475--1494},
  year			= {2020},
  mrnumber   = {4055529}
}

@article{Kim2026,
  author		= {Kim, J.},
  title			= {Application of some $L_{2}$ optimization to a discrete distribution},
  journal		= {Ann. Inst. Statist. Math.},
  volume		= {78},
  pages			= {43--67},
  year			= {2026}
}

@article{ Koul1970,
  author		= {Koul, H. L.},
  title			= {Some convergence theorems for ranks and weighted empirical cumulatives},
  journal		= {Ann. Math. Stat.},
  volume		= {41},
  number={5},
  pages			= {1768--1773},
  year			= {1970},
  mrnumber   = {0267631}
}

@book{Koul2002,
  author		= {Koul, H. L.},
  title			= {Weighted empirical process in nonlinear dynamic models},
  address		= {Berlin},
  publisher		= {Springer},
  year			= {2002},
  mrnumber = {1911855}

}

@article{ Millar1984,
  author		= {Millar, P. W.},
  title			= {A general approach to the optimality of minimum distance estimators},
  journal		= {Trans. Amer. Math. Soc.},
  volume		= {286},
  number={1},
  pages			= {377--418},
  year			= {1984}
}

@article{ Parr,
  author		= {Parr, W. C. and Schucany, W. R.},
  title			= {Minimum distance and robust estimation},
  journal		= {J. Am. Stat. Assoc.},
  volume		= {75},
  number={371},
  pages			= {616--624},
  year			= {1980},
  mrnumber   = {0590691}
}

@article{Wolfowitz1953 ,
  author		= {Wolfowitz, J.},
  title			= {Estimation by the minimum distance method},
  journal		= {Ann. Inst. Statisti. Math.},
  volume		= {5},
  pages			= {9--23},
  year			= {1953},
  mrnumber   = {0058931}
}

\edt

outlandish, outrageous, paltry